\documentclass[12pt,reqno]{amsart}
\usepackage{srcltx}
\newtheorem{theorem}{\bf Theorem}
\newtheorem{proposition}[theorem]{\bf Proposition}
\newtheorem{corollary}[theorem] {\bf Corollary}

\newtheorem{example}{\bf Example}

\newtheorem{remark}{\bf Remark}

\newtheorem{problem}[theorem]{Problem}

\DeclareMathOperator{\interior}{Int\,\!}

\usepackage{geometry,amssymb}
\newcommand{\Lh}{\widehat{L}}
\newcommand{\Dh}{\widehat{D}}
\newcommand{\la}{\lambda}

\begin{document}
\title[Bernstein and Markov-type inequalities for polynomials on $L_{p}(\mu)$ spaces.]
{Bernstein and Markov-type inequalities for polynomials on $L_{p}(\mu)$ spaces.}

\author[M. Chatzakou \and Y. Sarantopoulos]
{Marianna Chatzakou and Yannis Sarantopoulos}

\thanks{2010 Mathematics Subject Classification. Primary 41A17; Secondary 46G25, 47H60}

%\thanks{This work formed a part of the first author's Ph.D. thesis.}

\address[Marianna Chatzakou]
{Mathematics Department, Imperial College London\\
180 Queen's Gate, London, SW7 2AZ, UK.}
\email{m.chatzakou16@imperial.ac.uk}

\address[Yannis Sarantopoulos]
{Department of Mathematics, National Technical University\\
Zografou Campus 157 80, Athens, Greece.}
\email{ysarant@math.ntua.gr }

\maketitle
\begin{center}
{\it In memory of Professor Victor Lomonosov}
\end{center}

%%%%%%%%%%%%%%%%%%%%%%%%%%%%%%%%%%%%%%%%%%%%%%%%%%%%%%%%%%%%%%%%%
%%%%%                                                       %%%%%
%%%%%                     ABSTRACT                          %%%%%
%%%%%                                                       %%%%%
%%%%%%%%%%%%%%%%%%%%%%%%%%%%%%%%%%%%%%%%%%%%%%%%%%%%%%%%%%%%%%%%%

\begin{abstract}
In this work we discuss generalizations of the classical Bernstein and Markov type inequalities for polynomials and we present some new inequalities for the $k$th Fr\'echet derivative of homogeneous polynomials on real and complex $L_{p}(\mu)$ spaces. We also give applications to homogeneous polynomials and symmetric multilinear mappings in $L_{p}(\mu)$ spaces. Finally, Bernstein's inequality for homogeneous polynomials on both real and complex Hilbert spaces has been discussed.
\end{abstract}

%\maketitle

%%%%%%%%%%%%%%%%%%%%%%%%%%%%%%%%%%%%%%%%%%%%%%%%%%%%%%%%%%%%%%%%%
%%%%%                                                       %%%%%
%%%%%                     SECTION 1                         %%%%%
%%%%%                                                       %%%%%
%%%%%%%%%%%%%%%%%%%%%%%%%%%%%%%%%%%%%%%%%%%%%%%%%%%%%%%%%%%%%%%%%

\section{\bf Introduction and notation.}

We recall the basic definitions needed to discuss polynomials from $X$ into $Y$, where $X$ and $Y$ are real or complex Banach spaces. We denote by $B_X$ and $S_X$ the closed unit ball and the unit sphere of $X$ respectively. A map $P: X\rightarrow Y$ is a (continuous) {\em $m$-homogeneous polynomial} if there is a (continuous) symmetric $m$-linear mapping $L: X^m \rightarrow Y$ for which $P(x)=L(x, \ldots, x)$ for all $x\in X$. In this case it is convenient to write $P=\widehat{L}$. We let $\mathcal{P}(^{m}X; Y)$, $\mathcal{L}({}^{m}X; Y)$ and $\mathcal{L}^{s}({}^{m}X; Y)$ denote respectively the spaces of continuous $m$-homogeneous polynomials from $X$ into $Y$, the continuous $m$-linear mappings from $X$ into $Y$ and the continuous symmetric $m$-linear mappings from $X$ into $Y$. If $\mathbb{K}$ is the real or complex field we use the notations $\mathcal{P}({}^{m}X)$, $\mathcal{L}({}^{m}X)$ and $\mathcal{L}^{s}({}^{m}X)$ in place of $\mathcal{P}({}^{m}X; \mathbb{K})$, $\mathcal{L}({}^{m}X; \mathbb{K})$ and $\mathcal{L}^{s}({}^{m}X; \mathbb{K})$ respectively. More generally, a map
$P: X\rightarrow Y$ is a {\em continuous polynomial of degree $\leq m$} if
\[
P=P_0 +P_1 +\cdots +P_m\,,
\]
where $P_k \in\mathcal{P}({}^{k}X; Y)$, $1\leq k\leq m$, and $P_0: X\rightarrow Y$ is a constant function. The space of continuous polynomials from $X$ to $Y$ of degree at most $m$ is denoted by $\mathcal{P}_m(X; Y)$. If $Y=\mathbb{K}$, then we use the notation $\mathcal{P}_{m}(X)$ instead of $\mathcal{P}_{m}(X; \mathbb{K})$. We define the norm of a continuous (homogeneous) polynomial $P: X\rightarrow Y$ by
\[
\|P\|_{B_X}=\sup\{\|P(x)\|_Y:\ x\in B_X\}\,.
\]
Similarly, if $L: X^{m} \rightarrow Y$ is a continuous $m$-linear mapping we define its norm by
\[
\|L\|_{B_{X}^{m}}=\sup\{\|L(x_1, \ldots, x_m)\|_Y: \ x_1, \ldots, x_m\in B_X\}\,.
\]
When convenient we shall denote $\|L\|_{B_{X}^{m}}$ by $\|L\|$ and $\|P\|_{B_{X}}$ by $\|P\|$. Note that $\mathcal{P}({}^m X; Y)$ and $\mathcal{L}({}^m X; Y)$ are Banach spaces.

The classical Bishop-Phelps theorem \cite{BP} asserts that the collection of norm attaining continuous linear functionals on a Banach space $X$ is norm dense in $X^\ast := \mathcal{L}(^{1}X)$, the space of all continuous linear functionals on $X$. However, in contrast to the linear case, the set of norm attaining continuous symmetric $m$-linear forms ($m\geq 2$) on a Banach space $X$ is not generally norm dense in the Banach space of all continuous symmetric $m$-linear forms on $X$, and the set of norm-attaining continuous $m$-homogeneous polynomials on $X$ is not generally norm dense in the Banach space of all continuous $m$-homogeneous polynomials on $X$ \cite{AAP}. In fact, an example of a Banach space $X$ was given in \cite{AAP} such that the set of norm-attaining bilinear forms on $X\times X$ is not dense in the space of all continuous bilinear forms.  We refer to \cite{PST} for the relationship between the norm-attaining condition for a continuous homogeneous polynomial on a Banach space and the norm-attaining condition for its associated continuous symmetric multilinear form.

If $P\in\mathcal{P}_m(X; Y)$ and $x\in X$, then $D^{k}P(x)$, $2\leq k\leq m$, denotes the $k$th Fr\'echet derivative of $P$ at $x$. Recall that $D^{k}P(x)$ would be, in fact, a symmetric $k$-linear mapping on $X^{k}$, whose associated $k$-homogeneous polynomial will be represented by $\widehat{D}^{k}P(x)$. So, $\widehat{D}^{k}P(x):= \widehat{D^{k}P(x)}$. We just write $DP(x)$ for the first Fr\'echet derivative of $P$ at $x$. If $\widehat{L}\in\mathcal{P}(^{m}X; Y)$, for any vectors $x, y_{1}, \ldots ,y_{k}$ in $X$ and any $k\leq m$ the following identity (see for instance \cite[7.7 Theorem]{Chae}) holds
\begin{equation}\label{k-differential}
\frac{1}{k!}D^{k}\Lh(x)(y_{1}, \ldots, y_{k})={m\choose k}L(x^{m-k}, y_{1}, \ldots, y_{k})\,.
\end{equation}
In particular, for $x, y\in X$
\begin{equation}\label{k-homog-differential}
\frac{1}{k!}\Dh^{k}\Lh(x)y={m\choose k}L(x^{m-k} y^{k})
\end{equation}
and for $k=1$
\begin{equation}\label{1-differential}
D\Lh(x)y=\Dh\Lh(x)y=mL(x^{m-1} y)\,.
\end{equation}
Here, $L(x^{m-k} y^{k})$ denotes $L(\underbrace{x, \ldots, x}_{(m-k)}, \underbrace{y, \ldots, y}_{k})$. For general background on polynomials, we refer to \cite{Chae} and \cite{Dineen}.

Finally, observe that by composing $P\in\mathcal{P}_m(X; Y)$ with a given linear functional and applying the Hahn-Banach theorem, the upper bounds for  $\|D^kP(x)\|$ and $\|\widehat{D}^k P(x)\|$ to be determined are unchanged when $Y$ is replaced by $\mathbb{R}$ or $\mathbb{C}$. Therefore, in proving Bernstein and Markov-type inequalities on real or complex Banach spaces, without loss of generality we can restrict ourselves to scalar-valued polynomials.

Let $r_{n}(t)=\text{sign}(\sin2^n\pi t)$ be the $n$th Rademacher function on [0,1]. The Rademacher functions $(r_{n})$ form an orthonormal set in $L_{2}([0, 1], dt)$ where $dt$ denotes Lebesgue measure on [0,1]. The next formula expresses a well known {\it polarization formula} in a very convenient form (see \cite[Lemma 2]{Sarantopoulos1}):
\begin{equation}\label{polarformula}
L(x_1, \ldots, x_m) = \frac{1}{m!}\int_{0}^{1} r_{1}(t)\cdots r_{m}(t) \Lh\bigg(\sum_{n=1}^{m}r_{n}(t)x_{n}\bigg)\, dt\,.
\end{equation}
Therefore, each $\Lh\in\mathcal{P}({}^{m}X)$ is associated with a unique $L\in{\mathcal L}^{s}({}^{m} X)$ with the property that $\Lh(x)=L(x, \ldots, x)$. In many circumstances \cite{Davie-1,Davie-2,Peller,Varopoulos} it is of interest to compare the norm of $L\in\mathcal{L}^{s}(^{m}X)$ with the norm of  $\Lh\in\mathcal{P}({}^{m}X)$. It follows from \eqref{polarformula}(see \cite{Dineen}) that
\[
\|\Lh\|\leq \|L\|\leq \frac{m^{m}}{m!}\|\Lh\|\,,
\]
for every $L\in{\mathcal L}^{s}({}^{m} X)$. However, the right hand inequality can be tightened for many Banach spaces,  see for instance \cite{Dineen,Harris1,Sarantopoulos1}, and we call
\[
\mathbb{K}(m, X)=\inf \left\{M>0: \|L\|\leq M \|\Lh\|,\,\forall L\in\mathcal{L}^{s}({}^{m}X; \mathbb{K})\right\}
\]
the {\em $m$th polarization constant} of the Banach space $X$. We shall write $\mathbb{R}(m, X)$, $\mathbb{C}(m, X)$ instead of $\mathbb{K}(m, X)$, if the space $X$ is real, complex respectively.

For $L_{p}(\mu)$ spaces we also set
\[
\mathbb{K}(m, p)=\sup\{\mathbb{K}(m, L_{p}(\mu)): \text{$\mu$ is a measure}\}\,.
\]
It is an interesting fact that $\mathbb{K}(m, p)=\mathbb{K}(m, L_{p}(\mu))$, for any $\mu$ with $L_{p}(\mu)$ infinite-dimensional (we refer to \cite{Sarantopoulos1}).

%%%%%%%%%%%%%%%%%%%%%%%%%%%%%%%%%%%%%%%%%%%%%%%%%%%%%%%%%%%%%%%%%
%%%%%                                                       %%%%%
%%%%%                     SECTION 2                         %%%%%
%%%%%                                                       %%%%%
%%%%%%%%%%%%%%%%%%%%%%%%%%%%%%%%%%%%%%%%%%%%%%%%%%%%%%%%%%%%%%%%%

\section{\bf Bernstein-Markov inequalities for polynomials on Banach spaces.}

\subsection{\bf Bernstein-Markov inequalities for polynomials: classical results.}
Let $\mathcal{P}_n(\mathbb{R})$ be the set of all algebraic polynomials of degree at most $n$ with real coefficients. According to a well-known result of Bernstein \cite{Bernstein1}, if $p\in\mathcal{P}_n(\mathbb{R})$ and $\|p\|_{[-1,1]}:= \max_{-1\leq t\leq 1}|p(t)|\leq 1$ then
\begin{equation}\label{bernsteinpointwise}
|p^{\prime}(t)|\leq \frac{n}{\sqrt{1-t^2}}\,,\quad \forall t\in (-1,1)\,.
\end{equation}
It was proved by A. A. Markov that if $p\in\mathcal{P}_n(\mathbb{R})$ and $\|p\|_{[-1, 1]}\leq 1$, then
\begin{equation}\label{classicalrealfirst}
\|p^{\prime}\|_{[-1,1]}\leq n^2\,.
\end{equation}
A. A. Markov's original paper \cite{AMARKOV} dates back to 1889 and it is not readily accessible. For a modern exposition on this and other related topics we refer to \cite{Rahman}. Note that the upper bounds in \eqref{bernsteinpointwise} and \eqref{classicalrealfirst} are sharp since they are attained for the $n$th Chebyshev polynomial $T_n(t)$ (for certain values of $t$ in the case of \eqref{bernsteinpointwise}), where $T_n(t)$ is the polynomial agreeing $\cos(n\arccos{t})$ in the range
$-1<t<1$. Inequality \eqref{bernsteinpointwise} yields a better estimate for $|p^\prime(t)|$ when $t$ is not near $\pm 1$.

In the previous two inequalities we have estimates on the magnitude of the derivative of a polynomial, as compared to the polynomial itself. A related result is the following inequality known as Schur's inequality \cite[p. 233]{BORWEIN}:

For every $p\in\mathcal{P}_{n-1}(\mathbb{R})$,
\begin{equation}\label{schur}
\|p\|_{[-1,1]}\leq n\|p(t)\sqrt{1-t^2}\|_{[-1,1]}\,.
\end{equation}

Observe that Markov's inequality follows immediately from inequalities \eqref{bernsteinpointwise} and \eqref{schur}.

V. A. Markov (brother of A. A. Markov) considered the problem of determining exact bounds for the $k$th derivative of an algebraic polynomial. For $1\leq k\leq n$, if $p\in\mathcal{P}_n(\mathbb{R})$
and $\|p\|_{[-1,1]}\leq 1$, V. A. Markov \cite{VMARKOV} has shown that
\begin{equation}\label{classicalvm}
\|p^{(k)}\|_{[-1,1]}\leq T_n^{(k)}(1)=\frac{n^2(n^2-1^2)\cdots (n^2-(k-1)^2)}{1\cdot3\cdots (2k-1)}\,.
\end{equation}
S. N. Bernstein presented a shorter variational proof of \eqref{classicalvm} in $1938$ (see \cite{Bernstein2}). In 1938 Schaeffer and Duffin \cite{DUFFIN1} have given a rather simple proof of V. A. Markov's inequality. The key in their proof is the following generalization of Bernstein's inequality.

\vspace{.2cm}

\noindent {\bf Theorem A (A. C. Schaeffer \& R. J. Duffin \cite{DUFFIN1}).}
{\em If $p\in\mathcal{P}_n(\mathbb{R})$ with $\|p\|_{[-1,1]}\!\leq\! 1$ and $1\leq k\leq n$, then
\begin{equation}\label{pointwisekth}
|p^{(k)}(t)|^2\leq \big(T_n^{(k)}(t)\big)^2+ \big(S_n^{(k)}(t)\big)^2\,,\quad \forall t\in (-1,1)\,,
\end{equation}
where $S_n(t)$ is the is the polynomial agreeing $\sin(n\arccos{t})$ in the range $-1<t<1$.}

\vspace{.2cm}

In fact, if we define
\begin{equation}\label{defM_k}
\mathcal{M}_{k}(t):= \big(T_n^{(k)}(t)\big)^2+\big(S_n^{(k)}(t)\big)^2\,, \quad\forall\,t\in (-1, 1)\,,
\end{equation}
a close look at the proofs of Lemma 3 and Markoff's Theorem in \cite{DUFFIN1} reveals that the following result holds true.

\vspace{.2cm}

\noindent {\bf Theorem B (A. C. Schaeffer \& R. J. Duffin \cite{DUFFIN1}).}
{\em If $p\in\mathcal{P}_{n-k}(\mathbb{R})$, $1\leq k\leq n$, is such that $|p(t)|^2\leq\mathcal{M}_{k}(t)$\, $\forall\, t\in (-1, 1)$, then}
\[
\|p\|_{[-1, 1]}\leq T^{(k)}_n(1)=\frac{n^2(n^2-1^2)\cdots (n^2-(k-1)^2)}{1\cdot3\cdots (2k-1)}\,.
\]

\vspace{.2cm}

Notice that V. A. Markov's inequality \eqref{classicalvm} and Theorem B together imply Theorem A. Observe also that inequality \eqref{schur} (Schur's inequality) is a special case of Theorem B for $k=1$.

In studying extremal problems usually we normalize the set of polynomials, that is if $p\in\mathcal{P}_n(\mathbb{R})$ we take $|p(t)|\leq 1$ for $-1\leq t\leq 1$. In other words we require that the graph of $p$ is contained in the square $[-1,1]\times[-1, 1]$.

In the last twenty years extensions of the classical Bernstein and Markov-type inequalities to the multivariate case have been widely investigated. In \cite{Harris2} Harris considers the growth of the Fr\'echet derivatives of a polynomial on a normed space when the polynomial has restricted growth on the space. His main concern is with \emph{real} normed spaces. Using the technique of potential theory with external fields, improved estimates on Markov constants of \emph{homogeneous} polynomials over real normed spaces have been given in \cite{RevSar}. For the Markov inequality for multivariate polynomials we also refer to \cite{Kroo1} and \cite{Plesniak}. In $2012$ R\'ev\'esz \cite{Revesz} has given a survey on conjectures and results on the multivariate Bernstein inequality on convex bodies. For more polynomial inequalities in Banach spaces we refer to \cite{BARAN3}.

Finally, it is of importance how the pluripotential theory approach of Baran \cite{BARAN3} and the inscribed ellipse approach of Sarantopoulos \cite{Sarantopoulos2} relate. This is far from obvious and it was in fact unknown for long. However, in $2010$ it was fully clarified in \cite{BLMR}. In fact, Burns, Levenberg, Ma'u and R\'ev\'esz have shown in \cite{BLMR} that the ``inscribed ellipse method" of Sarantopoulos in \cite{Sarantopoulos2} to prove Bernstein-Markov inequalities and the ``pluripotential" proof of Bernstein-Markov inequalities due to Baran \cite{BARAN3} are equivalent.

\subsection{\bf Bernstein-Markov inequalities for polynomials on \textit{real} Hilbert spaces.}

Let $P$ be a polynomial of degree at most $n$ with real coefficients on $\ell_2^m$, the $m$-dimensional Euclidean space $\left(\mathbb{R}^{m}, \langle\cdot, \cdot\rangle\right)$. If $\|P\|\leq 1$ and $\|x\|_2<1$, the first sharp Bernstein and Markov-type inequalities were obtained in 1928 by Kellogg \cite{Kellogg}:
\begin{equation}\label{kellogg}
\|\nabla P(x)\|_{2}\leq\min\bigg\{\frac{n}{\sqrt{1-\|x\|_{2}^{2}}},\, n^{2}\bigg\}
\end{equation}
In other words, if $D_{y}P(x)=DP(x)y=\langle\nabla P(x), y\rangle$ is the directional derivative of $P$ at $x$, in the direction of the unit vector $y$, then the maximum of the absolute value of $D_{y}P(x)$ in any direction $y$ is just the maximum of the magnitude of the gradient of the polynomial and is dominated by the smaller of the two numbers $n/\sqrt{1-\|x\|_2^2}$ and $n^2$. In fact, Kellogg has derived \eqref{kellogg} by showing (see Theorem V in \cite{Kellogg}) that the tangential derivatives of $P$ on the unit sphere $S_{\mathbb{R}^m}$ cannot exceed $n$ in absolute value.

If $K$ is a smooth compact algebraic curve in ${\mathbb R}^2$ and $P$ is a polynomial of degree $\leq n$ in two variables, Bos \emph{et al.} \cite{BOS} have shown that
\[
\|D_TP\|_K\leq Mn\|P\|_{K}\,,
\]
where $D_TP$ denotes tangential derivative of $P$ along $K$, $\|P\|_{K}:= \sup|P|(K)$ and $M>0$ is a constant depending only on $K$. If $K$ is the unit circle, the previous inequality with $M=1$ is just Kellogg's result. For a discussion on this last inequality and for some other related results see \cite{BARAN2} and \cite{Fefferman}.

Harris \cite{Harris1} has extended Kellogg's argument and in the case of a \textit{real} Hilbert space
$(H, \langle\cdot, \cdot\rangle)$, if $P\in\mathcal{P}_{n}(H)$ and $\|P\|\leq 1$, he has obtained the following generalization of (\ref{kellogg}):
\[
|DP(x)y|\leq\min\left\{n\left[\frac{1-\|x\|^2+\langle x, y\rangle^2}{1-\|x\|^2}(1-P(x)^2)\right]^{1/2},
\, n^2\right\}\,,
\]
for all $\|x\|<1$ and $y\in S_H$.

The generalization of Markov's inequality for any derivative of a polynomial on a {\it real} Hilbert space was given in \cite{Gustavo2}. The proof relies on the following extension of Theorem A for polynomials on a {\it real} Hilbert space and the generalization of Theorem B for polynomials on any {\it real} Banach space. Recall that $\mathcal{M}_{k}(t)$ is given by \eqref{defM_k}.

\begin{theorem}\cite[Theorem 4]{Gustavo2}
If $\left(H, \langle\cdot, \cdot\rangle\right)$ is a {\it real} Hilbert space, $P\in {\mathcal P}_n(H)$ with $\|P\|\leq 1$ and $1\leq k\leq n$, then
\[
\|D^{k}P(x)\|^2=\|\widehat{D}^{k}P(x)\|^2\leq\mathcal{M}_{k}(\|x\|)\,,
\]
for every $x\in H$, $\|x\|<1$.
\end{theorem}

\begin{theorem}\cite[Lemma 1]{Gustavo2}
If $X$ is a {\it real} Banach space and $P\in\mathcal{P}_{n-k}(X)$, $n\geq k$, is such that $|P(x)|^2\leq\mathcal{M}_{k}(\|x\|)$, $\forall\, \|x\|<1$, then $|P(x)|\leq T_{n}^{(k)}(1)$, $\forall\, \|x\|\leq 1$.
\end{theorem}

Now, the generalization of Markov's inequality \eqref{classicalvm} on a {\it real} Hilbert space follows immediately from the previous two theorems.

\begin{theorem}$(${\bf V. A. Markov's theorem}$)$\cite[Theorem 5]{Gustavo2}
If $\left(H, \langle\cdot, \cdot\rangle\right)$ is a {\it real} Hilbert space, $P\in\mathcal{P}_n(H)$ with $\|P\|\leq 1$ and $1\leq k\leq n$, then
\[
\|D^{k}P(x)\|=\|\widehat{D}^{k}P(x)\|\leq T_n^{(k)}(1)=\frac{n^2(n^2-1^2)\cdots (n^2-(k-1)^2)}
{1\cdot3 \cdots (2k-1)}\,,
\]
for every $x\in B_H$.
\end{theorem}

\subsection{\bf Bernstein-Markov inequalities for polynomials on \textit{real} Banach spaces.}

Let $K\subset\mathbb{R}^{m}$ be a convex body, i.e. a convex compact set with non-empty interior. If $u$ is a unit vector in $\mathbb{R}^{m}$ then there are precisely two support hyperplanes to $K$ having $u$ for a normal vector. The distance $w(u)$ between these parallel support hyperplanes is the width of $K$ in the direction of $u$. The {\em minimal width} of $K$ is
\[
w(K):= \min_{\|u\|_{2}=1} w(u)\,.
\]
For general background on convexity, we refer to \cite{EGGLESTON}. If $P\in\mathcal{P}_{n}(\mathbb{R}^{m})$ with $\|P\|_K:=\max_{x\in K}|P(x)|$, Wilhelmsen \cite{Wilhelmsen} has shown that
\begin{equation}\label{wilhelmsen-1}
\|\nabla P\|_K=\max_{x\in K}\|\nabla P(x)\|_{2}\leq\frac{4n^2}{w(K)}\|P\|_{K}\,.
\end{equation}
Since $w(B_{\ell_2^m})=2$, the constant in (\ref{wilhelmsen-1}) is two times the constant in (\ref{kellogg}).

Consider now the case where $K\subset\mathbb{R}^{m}$ is a centrally symmetric convex body with center at the origin, in other words $K$ is invariant under $x\mapsto -x$. We call $K$ a ball. A ball $K$ is the unit ball of a unique Banach norm $\|\cdot\|_K$ defined by
\[
\|x\|_{K}=\inf\{t>0: \ x/t\in K\},\quad x\in\mathbb{R}^{m}\,.
\]
If $P\in\mathcal{P}_{n}(\mathbb{R}^m)$, $x\in\interior{K}$ and $y\in S_{\mathbb{R}^m}$, the next sharp Bernstein and Markov-type inequalities follow from the work of Sarantopoulos\cite{Sarantopoulos2}:
\begin{align}
|DP(x)y| &\leq \frac{2n}{w(K)\sqrt{1-\|x\|_K^2}}\|P\|_{K}\,, \label{sarantopoulos1}\\
\|\nabla P\|_K &\leq\frac{2n^2}{w(K)}\|P\|_{K}\,. \label{sarantopoulos2}
\end{align}
In fact, if $X$ is a \textit{real} Banach space and $P\in\mathcal{P}_{n}(X)$, $\|P\|\leq 1$, for the first Fr\'echet derivative of $P$ it has been proved in \cite{Sarantopoulos2} that

\begin{equation}\label{firstgen}
\|DP(x)\|\leq\min\left\{n\frac{\sqrt{1-P(x)^2}}{\sqrt{1-\|x\|^2}},\, n^{2}\right\},\quad\text{for every $\|x\|<1$}\,.
\end{equation}

Using methods of several complex variables, inequalities (\ref{sarantopoulos1}) and (\ref{sarantopoulos2}) were proved independently by Baran \cite{BARAN1}. For non-symmetric convex bodies, Kro\'o and R\'ev\'esz \cite{KR} have derived a Bernstein-type inequality and they have shown \eqref{wilhelmsen-1} with constant $(4n^2-2n)/w(K)$. They have also achieved a further improvement on the Markov constant in case $K$ is a triangle in ${\mathbb R}^2$. But, as it has been shown in \cite{Bialas}, in the non-symmetric case the Markov constant in \eqref{wilhelmsen-1} has to be larger than $2$.

Finally, a proof of Markov's inequality for any derivative of a polynomial on a {\it real} Banach space was given by Skalyga \cite{Skalyga1} in 2005 and additional discussion is given in \cite{Skalyga2}. In 2010 Harris \cite{Harris4} has given another proof which depends on a Lagrange interpolation formula for the Chebyshev nodes and a Christoffel-Darboux identity for the corresponding bivariate Lagrange polynomials \cite{Harris3}.

\begin{theorem}$(${\bf V. A. Markov's theorem}$)$\cite{Skalyga1, Skalyga2, Harris3}
If $X$ is a {\it real} Banach space, $P\in\mathcal{P}_n(X)$ with $\|P\|\leq 1$ and $1\leq k\leq n$, then
\[
\|\widehat{D}^{k}P(x)\|\leq T_n^{(k)}(1)\,,
\]
for all $x\in X$, $\|x\|\leq 1$.
\end{theorem}

Kro\'o \cite{Kroo2} has derived certain Bernstein-Markov inequalities for multivariate polynomials on convex and star-like domains in finite dimensional {\em real} $L_{p}(\mu)$ spaces, $1\leq p<\infty$. In \cite[Theorem 6]{EI-1} Eskenazis and Ivanisvili have obtained dimension independent Bernstein-Markov inequality in Gauss space. That is, for each $1\leq p<\infty$ there is a constant $C_{p}>0$ such that for any $k\geq 1$ and all polynomials $P$ on $\mathbb{R}^{k}$
\[
\|\nabla P\|_{L_{p}(d\gamma_k)}\leq C_{p}(\deg{P})^{\alpha}\|P\|_{L_{p}(d\gamma_k)}\,,
\]
where $d\gamma_{k}(x)=\frac{e^{-\|x\|_{2}^{2}/2}}{\sqrt{(2\pi)^{k}}}\,dx$ is the standard Gaussian measure on $\mathbb{R}^{k}$, $\alpha=\frac{1}{2}+\frac{1}{\pi}\arctan\big(\frac{|p-2|}{2\sqrt{p-1}}\big)$ and
\[
\|\nabla P\|_{L_{p}(d\gamma_k)}:= \left(\int_{\mathbb{R}^{k}} \bigg(\sum_{j=1}^{k} (\partial_{j}P)^{2}(x)\bigg)^{p/2}\, d\gamma_{k}(x)\right)^{1/p}\,.
\]
We also refer to \cite{EI-2} for polynomial inequalities on the Hamming cube.

%%%%%%%%%%%%%%%%%%%%%%%%%%%%%%%%%%%%%%%%%%%%%%%%%%%%%%%%%%%%%%%%%
%%%%%                                                       %%%%%
%%%%%                     SECTION 3                         %%%%%
%%%%%                                                       %%%%%
%%%%%%%%%%%%%%%%%%%%%%%%%%%%%%%%%%%%%%%%%%%%%%%%%%%%%%%%%%%%%%%%%

\section{\bf Bernstein-Markov inequalities for homogeneous polynomials on $L_{p}(\mu)$ spaces.}

\subsection{\bf Bernstein and Markov-type estimates for homogeneous polynomials on $L_{p}(\mu)$ spaces.}

In the case of a continuous homogeneous polynomial $P\in\mathcal{P}(^{m}X; Y)$, where $X$ and $Y$ are {\it real} Banach spaces, the constant $c_{m, k}$ in V. A. Markov's inequality
\[
\|\widehat{D}^{k}P\|\leq c_{m, k}\|P\|\,,
\]
can be improved and is considerably better than $T_{m}^{(k)}(1)$.

For continuous homogeneous polynomials on {\it real} Banach spaces we have the following Bernstein and Markov-type inequalities.

\begin{theorem}\cite[Theorem 3]{Sarantopoulos2}\label{general-bern-ineq}
If $X$ is a {\it real} Banach space and $\Lh: X\rightarrow \mathbb{R}$ is a continuous $m$-homogeneous polynomial, then we have the following Bernstein-type inequalities
\begin{align}
(a)\qquad \|\Dh^{k}\Lh(x)\| &\leq {m\choose k}\frac{k!} {\big(\sqrt{1-\|x\|^{2}}\big)^{k}}\|\Lh\|\label{general-bern-ineq-1}\\
\text{and}\notag\\
(b)\qquad \|D^{k}\Lh(x)\| &\leq {m\choose k} \frac{k^{k}}{\big(\sqrt{1-\|x\|^{2}}\big)^{k}}\|\Lh\|\,,\label{general-bern-ineq-2}
\end{align}
for any $\|x\|<1$ and $k\leq m$.
\end{theorem}

\begin{corollary}\cite[Corollary]{Sarantopoulos2}\label{general-markov-ineq}
If $X$ is a {\it real} Banach space and $\Lh: X\rightarrow \mathbb{R}$ is a continuous $m$-homogeneous polynomial, then we have the following Markov-type inequalities
\begin{align}
(a)\qquad \|\Dh^{k}\Lh(x)\| &\leq {m\choose k} \frac{k!m^{m/2}}{(m-k)^{(m-k)/2}k^{k/2}}\|\Lh\|\label{general-markov-ineq-1}\\
\text{and}\notag\\
(b)\qquad\|D^{k}\Lh(x)\| &\leq {m\choose k}\frac{m^{m/2}k^{k/2}}{(m-k)^{(m-k)/2}}\|\Lh\|\,,
\end{align}
for any $\|x\|\leq 1$ and $k\leq m$.
\end{corollary}

Now we prove Bernstein and Markov-type inequalities for homogeneous polynomials on any {\it complex} $L_{p}(\mu)$ space, $1\leq p\leq\infty$. For the proof we need the \textit{generalized Clarkson inequality}
\begin{equation}\label{Clarkson-gen}
\left(\|x_1+x_2\|_{p}^{\la'}+\|x_1-x_2\|_{p}^{\la'}\right)^{1/\la'}
\leq 2^{1/\la'}\left(\|x_1\|_{p}^{\la}+\|x_2\|_{p}^{\la}\right)^{1/\la}\,,
\end{equation}
where $x_1, x_2\in L_{p}(\mu)$ and $1\leq\la\leq\min\{p, p'\}$. Here, as usual, $\la'=\la/(\lambda-1)$ and $p'=p/(p-1)$ are the conjugate exponents of $\la$ and $p$ respectively. Inequality \eqref{Clarkson-gen} is a special case for $m=2$ of the following $L_p$-inequality
\begin{equation}\label{WW-ineq}
\bigg(\int_{0}^{1} \big\|\sum_{i=1}^{m} r_{i}(t)x_{i}\big\|_{p}^{\la'}\, dt\bigg)^{1/\la'} \leq\big(\sum_{i=1}^{m} \|x_{i}\|_{p}^{\la}\big)^{1/\la}\,,
\end{equation}
for $x_{i}\in L_{p}(\mu)$, $1\leq i\leq m$ and $1\leq\la\leq\min\{p, p'\}$. We refer to \cite{WW} for this and other similar $L_{p}$ inequalities.
Setting $\la=p$ or $\lambda=p'$, inequality \eqref{Clarkson-gen} gives the classical Clarkson inequalities:
\begin{align*}
\left(\|x_1+x_2\|_{p}^{p'}+\|x_1-x_2\|_{p}^{p'}\right)^{1/p'} &\leq 2^{1/p'} \left(\|x_1\|_{p}^{p}+\|x_2\|_{p}^{p}\right)^{1/p}\,,\quad 1\leq p\leq 2\,,\\
\left(\|x_1+x_2\|_{p}^{p}+\|x_1-x_2\|_{p}^{p}\right)^{1/p} &\leq 2^{1/p} \left(\|x_1\|_{p}^{p'}+\|x_2\|_{p}^{p'}\right)^{1/p'}\,,\quad 2\leq p\leq\infty\,.
\end{align*}

\begin{theorem}\label{bern-ineq}
Let $\Lh: L_{p}(\mu)\rightarrow \mathbb{C}$ be a continuous $m$-homogeneous polynomial, $m\geq 2$, on the {\it complex} $L_{p}(\mu)$ space. If $m'$ and $p'$ are the conjugate exponents of $m$ and $p$ respectively, for $k\leq m$ and every $x\in L_{p}(\mu)$, $\|x\|_{p}<1$, we have the following Bernstein-type inequalities
\begin{equation}\label{bern-ineq-1}
\|\Dh^{k}\Lh(x)\|\leq \begin{cases} \frac{k!}{(1-\|x\|_p^{p})^{k/p}}\|\Lh\| &\text{$1\leq p\leq m'$ ,}\\
\frac{k!}{(1-\|x\|_p^{m'})^{k/m'}}\|\Lh\| &\text{$m'\leq p\leq m$ ,}\\
\frac{k!}{(1-\|x\|_p^{p'})^{k/p'}}\|\Lh\| &\text{$m\leq p\leq\infty$ .}
                      \end{cases}
\end{equation}
\end{theorem}

\begin{proof}
\underline{\it{1st case}}: Let $1\leq p\leq m'\Leftrightarrow m\leq p'\leq\infty$ or $m\leq p \leq \infty \Leftrightarrow 1\leq p'\leq m'$. If $\la=\min\{p, p'\}$, then $\la=p$, for $1\leq p\leq m'$ and $\la=p'$, for $m\leq p\leq\infty$. For every $x, y\in L_{p}(\mu)$, $\|x\|_{p}<1$, $\|y\|_{p}=1$ and every $z\in\mathbb{C}$ put
\[
q(z):= \Lh\big(x+(1-\|x\|_p^{\la})^{1/\la}yz\big) +(-1)^{k}\Lh\big(x-(1-\|x\|_p^{\la})^{1/\la}yz\big)\,.
\]
Then $q$ is a polynomial of degree $\leq m$ on $\mathbb{C}$ with
\[
|q(z)|\leq \|\Lh\|\left\{\|x+(1-\|x\|_p^{\la})^{1/\la}yz)\|_p^{m} +\|x-(1-\|x\|_p^{\la})^{1/\la}yz)\|_p^{m}\right\}\,.
\]
Applying H\"older's inequality first and then Clarkson's inequality \eqref{Clarkson-gen}, for $|z|\leq 1$ we have
\begin{eqnarray*}
\lefteqn{\|x+(1-\|x\|_p^{\la})^{1/\la}yz)\|_p^{m}+\|x-(1-\|x\|_p^{\la})^{1/\la}yz)\|_p^{m}}\\
& &\leq 2^{1-m/\la'}\left\{\|x+(1-\|x\|_p^{\la})^{1/\la}yz)\|_p^{\la'}+\|x-(1-\|x\|_p^{\la})^{1/\la}yz)\|_p^{\la'}\right\}^{m/\la'}\\
& &\leq 2^{1-m/\la'}\cdot 2^{m/\la'}\left\{\|x\|_p^{\la} +\|(1-\|x\|_p^{\la})^{1/\la}yz)\|_p^{\la}\right\}^{m/\la}\\
& &\leq 2\left\{\|x\|_p^{\la}+(1-\|x\|_p^{\la})\right\}^{m/\la}=2\,.
\end{eqnarray*}
Hence, $|q(z)|\leq 2\|\Lh\|$, for every $|z|\leq 1$ and by the Cauchy estimates $|q^{(k)}(0)|\leq k!\cdot 2\|\Lh\|$. Since
$q^{(k)}(0)=2(1-\|x\|_p^{\la})^{k/\la}\Dh^{k}\Lh(x)y$, we have
\[
\|\Dh^{k}\Lh(x)\|\leq\frac{k!}{(1-\|x\|_p^{\la})^{k/\la}}\|\Lh\|
\]
and this proves the first and the third estimate in \eqref{bern-ineq-1}.

\underline{\it{2nd case}}: Let $m'\leq p\leq m\Leftrightarrow m'\leq p'\leq m$. For every $x, y\in L_{p}(\mu)$, $\|x\|_{p}<1$, $\|y\|_{p}=1$ and every $z\in\mathbb{C}$ put
\[
q(z):= \Lh\big(x+(1-\|x\|_p^{m'})^{1/m'}yz\big)+(-1)^{k}\Lh\big(x-(1-\|x\|_p^{m'})^{1/m'}yz\big)\,.
\]
Then $q$ is a polynomial of degree $\leq m$ on $\mathbb{C}$ with
\[
|q(z)|\leq \|\Lh\|\left\{\|x+(1-\|x\|_p^{m'})^{1/m'}yz)\|_p^{m}+\|x-(1-\|x\|_p^{m'})^{1/m'}yz)\|_p^{m}\right\}\,.
\]
For every $|z|\leq 1$, Clarkson's inequality \eqref{Clarkson-gen} for $\lambda=m'\leq\min\{p, p'\}$ implies
\begin{eqnarray*}
\lefteqn{\|x+(1-\|x\|_p^{m'})^{1/m'}yz)\|_p^{m}+\|x-(1-\|x\|_p^{m'})^{1/m'}yz)\|_p^{m}}\\
& &\leq 2\left\{\|x\|_p^{m'}+\|(1-\|x\|_p^{m'})^{1/m'}yz)\|_p^{m'}\right\}^{m/m'}\\
& &\leq 2\left\{\|x\|_p^{m'}+(1-\|x\|_p^{m'})\right\}^{m/m'}=2\,.
\end{eqnarray*}
Hence, $|q(z)|\leq 2\|\Lh\|$, for every $|z|\leq 1$ and by the Cauchy estimates $|q^{(k)}(0)|\leq k!\cdot 2\|\Lh\|$. Since $q^{(k)}(0)=2(1-\|x\|_p^{m'})^{k/m'}\Dh^{k}\Lh(x)y$, we have
\[
\|\Dh^{k}\Lh(x)\|\leq\frac{k!}{(1-\|x\|_p^{m'})^{k/m'}}\|\Lh\|
\]
which is the second estimate in \eqref{bern-ineq-1}.
\end{proof}

\begin{remark}
Harris \cite[Theorem $10$]{Harris1} has proved a Bernstein-type inequality for a holomorphic function $h$ satisfying certain conditions on a {\em complex} $L_{p}(\mu)$ space, $1\leq p\leq\infty$. In particular, if $h$ is a homogeneous polynomial $\Lh$ of degree $m=2k$ on some $L_{p}(\mu)$ space, he gives an upper bound for the norm $\|\Dh^{k}\Lh(x)\|$, for all $x\in L_{p}(\mu)$, $\|x\|_{p}\leq 1/2$.
\end{remark}

\begin{proposition}\label{complexmarkov-ineq}
Let $\Lh: L_{p}(\mu)\rightarrow \mathbb{C}$ be a continuous $m$-homogeneous polynomial, $m\geq 2$, on the {\it complex} $L_{p}(\mu)$ space. If $m'$ and $p'$ are the conjugate exponents of $m$ and $p$ respectively, for $k\leq m$ we have the following Markov-type inequality
\begin{equation}\label{markov-ineq-complex}
\|\Dh^{k}\Lh\|\leq C_{k, m}\|\Lh\|\,,
\end{equation}
where
\begin{equation}\label{markov-ineq-1}
C_{k, m}= \begin{cases} \frac{k!m^{m/p}}{(m-k)^{(m-k)/p}k^{k/p}} &\text{$1\leq p\leq m'$ ,}\\
\frac{k!m^{m/m'}}{(m-k)^{(m-k)/m'}k^{k/m'}} &\text{$m'\leq p\leq m$ ,}\\
\frac{k!m^{m/p'}}{(m-k)^{(m-k)/p'}k^{k/p'}} &\text{$m\leq p\leq\infty$ .}
          \end{cases}
\end{equation}
In the case $1\leq p\leq m'$ the estimate is best possible.
\end{proposition}

\begin{proof}
Consider the case $1\leq p\leq m'$ or $m\leq p\leq\infty$. If $\la=\min\{p, p'\}$ and $x\in L_{p}(\mu)$, $\|x\|_{p}<1$, from the previous theorem
\[
\|\Dh^{k}\Lh(x)\|\leq\frac{k!}{(1-\|x\|_p^{\la})^{k/\la}}\|\Lh\|
\]
and so for $\|x\|_{p}\leq 1$ and $0<r<1$ we have
\[
r^{m-k}\|\Dh^{k}\Lh(x)\|=\|\Dh^{k}\Lh(rx)\|\leq\frac{k!}{(1-r^{\la})^{k/\la}}\|\Lh\|\,.
\]
Therefore,
\[
\|\Dh^{k}\Lh(x)\|\leq\frac{k!}{r^{m-k}(1-r^{\la})^{k/\la}}\|\Lh\|\,,\quad\text{for $\|x\|_{p}\leq 1$ and $0<r<1$}\,.
\]
Observe that
\[
\min_{0<r<1} \frac{1}{r^{m-k}(1-r^{\la})^{k/\la}}=\frac{m^{m/\la}}{(m-k)^{(m-k)/\la}k^{k/\la}}
\]
and the minimum is attained for $r=\left(\frac{m-k}{m}\right)^{1/\la}$. Hence,
\[
\|\Dh^{k}\Lh\|=\sup_{\|x\|_{p}\leq 1} \|\Dh^{k}\Lh(x)\| \leq\frac{k!m^{m/\la}}{(m-k)^{(m-k)/\la}k^{k/\la}}\|\Lh\|\,.
\]
Similar is the proof of the middle estimate in \eqref{markov-ineq-1}. Sharpness in the case $1\leq p\leq m'$ will follow from the next Example \ref{markov-eq}.
\end{proof}

Observe that in the case $1\leq p\leq m'$ the first inequality in \eqref{markov-ineq-1} also follows from a special case of \cite[Theorem 1]{Sarantopoulos1}.

\begin{example}\label{markov-eq}
Consider the symmetric $m$-linear form $L$ on the space of $p$-summable sequences $\ell_{p}$ given by
\[
L(x_1, \ldots, x_m)=\frac{1}{m!}\sum_{\sigma\in S_m} x_{1\sigma(1)}\cdots x_{m\sigma(m)}\,,
\]
where $x_i=(x_{in})_{n=1}^{\infty}$, $i=1, \ldots, m$, and $S_m$ is the set of permutations of the first $m$ natural numbers. Then, $\Lh(u)=u_{1}\cdots u_{m}$, $u=(u_{i})$, is the $m$-homogeneous polynomial associated to $L$. If $(e_i)$ is the standard unit vector basis of $\ell_{p}$, for the unit vectors
\[
x=\frac{1}{(m-k)^{1/p}}(e_1+\cdots +e_{m-k})\quad\text{and}\quad y=\frac{1}{k^{1/p}}(e_{m-k+1}+\cdots +e_{m})
\]
in $\ell_{p}$ we can easily verify (see \cite[\emph{Example 1}]{Sarantopoulos1}) that
\[
|L(x^{m-k} y^{k})|=\frac{(m-k)!k!}{(m-k)^{(m-k)/p}k^{k/p}}\cdot\frac{m^{m/p}}{m!}\|\Lh\|\,.
\]
Observe that
\[
|\Lh(u)|=\{|u_{1}|^{p}\cdots|u_{m}|^{p}\}^{1/p}\leq\left\{\frac{|u_{1}|^{p}+\cdots +|u_{m}|^{p}}{m}\right\}^{m/p}
\]
by the arithmetic-geometric mean inequality and so $\|\Lh\|\leq 1/m^{m/p}$. In fact $\|\Lh\|=1/m^{m/p}$ since for the unit vector $v=(v_{i})$ in $\ell_{p}$, with $v_{i}=m^{-1/p}$ for $1\leq i\leq m$ and $v_{i}=0$ for $i>m$, $|\Lh(v)|=1/m^{m/p}$.  Therefore, identity \eqref{k-homog-differential} implies
\[
|\Dh^{k}\Lh(x)y|=k!{m\choose k} |L(x^{m-k} y^k)|=\frac{k!m^{m/p}}{(m-k)^{(m-k)/p}k^{k/p}}\|\Lh\|\,.
\]
\end{example}

Now we give Markov-type estimates in the case of \emph{real} $L_{p}(\mu)$ spaces. For this we need some results related to complexification of \emph{real} Banach spaces, polynomials and multilinear maps,
see \cite{Gustavo1}.

A \emph{complex} vector space $\widetilde{X}$ is a \emph{complexification} of
a \emph{real} vector space $X$ if the following two conditions hold:
\begin{itemize}
\item[(i)] there is a one-to-one real-linear map $j: X\rightarrow \widetilde{X}$ and

\item[(ii)] complex-span$\big(j(X)\big)=\widetilde{X}$.
\end{itemize}
If $X$ is a real vector space, we can make $X\times X$ into a complex vector space by
defining
\begin{align*}
(x,y)+(u,v) &:= (x+u, y+v)\quad\forall x,y,u,v\in X\,, \\
(\alpha+i\beta)(x, y) &:= (\alpha x-\beta y, \beta x+\alpha y)\quad\forall x,y\in X,\quad\forall \alpha, \beta \in\mathbb{R}.
\end{align*}
The map $j: X\rightarrow X\times X$; $x\mapsto (x, 0)$ clearly satisfies conditions (i) and (ii) above, and so this complex vector space is a complexification of $X$. It is convenient to denote it by
\[
\widetilde{X}=X\oplus iX\,.
\]
If $X$ is a real-valued $L_{p}(\mu)$-space, the complexification procedure yields the corresponding complex-valued space. Since $X=L_{p}(\mu)$ is actually a Banach lattice, the norm on $\widetilde{X}$
can be specified by
\[
\|(x, y)\|=\|(|x|^2+|y|^2)^{1/2}\|\,, \qquad \forall x,y\in X.
\]
Bochnak and Siciak (see \cite[Theorem 3]{BS}) observed that when $X$ is a real Banach space, each $L\in\mathcal{L}(^{m}X; \mathbb{R})$ has a unique complex extension $\widetilde{L}\in\mathcal{L} (^{m}\widetilde{X}; \mathbb{C})$, defined by the formula
\[
\widetilde{L}(x_{1}^{0}+ ix_{1}^{1},\ldots ,x_{m}^{0}+ix_{m}^{1})=\sum i^{\sum_{j=1}^{m}{\epsilon}_{j}}
L(x_{1}^{{\epsilon}_{1}},\ldots ,x_{m}^{{\epsilon}_{m}}),
\]
where $x_{k}^{0}, x_{k}^{1}$ are vectors in $X$, and the summation is extended over the $2^m$ independent choices of $\epsilon_{k}=0, 1\ (1\leq k\leq m)$. The norm of $\widetilde{L}$ depends on the norm used on $\widetilde{X}$, but continuity is always assured.

In the context of polynomials (see also \cite[p.313]{Taylor}), any $P\in\mathcal{P}(^{m}X; \mathbb{R})$ has a unique complex extension $\widetilde{P}\in\mathcal{P}(^{m}\widetilde{X}; \mathbb{C})$, given by the formula
\begin{equation*}
\widetilde{P}(x+iy)=\sum_{k=0}^{[\frac{m}{2}]}(-1)^{k}{\binom {m}{2k}}L(x^{m-2k}y^{2k})
+i\sum_{k=0}^{[\frac{m-1}{2}]}(-1)^{k}{\binom {m}{2k+1}}L(x^{m-(2k+1)}y^{2k+1})
\end{equation*}
for $x, y$ in $X$, where $P:=\Lh$ for some $L\in\mathcal{L}^{s}(^{m}X; \mathbb{R})$. Here also $\widetilde{P}=\widehat{\widetilde{L}}$.

If $\widetilde{X}$ is the complexification of a real Banach space $X$, each $L\in\mathcal{L}^{s}(^{m}X; \mathbb{R})$ has a unique complex extension $\widetilde{L}\in\mathcal{L}^{s}(^{m}\widetilde{X}; \mathbb{C})$ with $\|L\|\leq \|\widetilde{L}\|$ and $\|P\|\leq \|\widetilde{P}\|$, where $P=\Lh$. We also have \cite[Proposition 18]{Gustavo1}
\begin{equation}\label{realvscomplex}
\|\widetilde{P}\|\leq 2^{m-1}\|P\|\quad\text{and}\quad\|\widetilde{L}\|\leq 2^{m-1}\|L\|\,.
\end{equation}

\begin{proposition}\label{realmarkov-ineq}
Let $\Lh: L_{p}(\mu)\rightarrow \mathbb{R}$ be a continuous $m$-homogeneous polynomial, $m\geq 2$, on the \emph{real} $L_{p}(\mu)$ space. Then, for $k\leq m$ we have the following Markov-type inequality
\begin{equation}\label{markov-ineq-real}
\|\Dh^{k}\Lh\|\leq 2^{m-1}C_{k, m}\|\Lh\|\,,
\end{equation}
where $C_{k, m}$ are the estimates in \eqref{markov-ineq-1}.
\end{proposition}

\begin{proof}
Let $P=\Lh\in\mathcal{P}(^{m}L_{p}(\mu); \mathbb{R})$. If $\widetilde{P}\in\mathcal{P}(^{m}L_{p}(\mu); \mathbb{C})$ is the unique extension of $P$ on the complex $L_{p}(\mu)$-space, it follows from \eqref{markov-ineq-complex} that
\[
\|\Dh^{k}\widetilde{P}\|\leq C_{k, m}\|\widetilde{P}\|\,.
\]
If we use the first inequality in \eqref{realvscomplex}, we have
\[
\|\Dh^{k}\Lh\|=\|\Dh^{k}P\|\leq\|\Dh^{k}\widetilde{P}\|\leq 2^{m-1}C_{k, m}\|\Lh\|\,.
\]
\end{proof}
The estimate in \eqref{markov-ineq-real} is far from optimal.

\subsection{\bf An application: Polarization constants of $L_{p}(\mu)$ spaces.}

Let $L\in{\mathcal L}^{s}(^{m} L_{p}(\mu))$. Consider first the case $1\leq p\leq m'$. Using formula \eqref{1-differential}, from inequality \eqref{markov-ineq-complex} and the estimate in \eqref{markov-ineq-1} with $k=1$ of Proposition \ref{complexmarkov-ineq} we have
\[
|L(x^{m-1} y)|=\frac{1}{m}|D\Lh(x)y|\leq\frac{m^{m/p-1}}{(m-1)^{(m-1)/p}}\|\Lh\|\,.
\]
Now an induction on $m$ implies that
\[
\|L\|\leq\frac{m^{m/p}}{m!}\|\Lh\|\,,\quad\text{for every $L\in{\mathcal L}^{s}(^{m} L_{p}(\mu))$, $1\leq p\leq m'$}
\]
and so $\mathbb{C}(m, p)\leq m^{m/p}/m!$. If $m\leq p\leq\infty$, using the estimate in \eqref{markov-ineq-1} a similar argument shows that $\mathbb{C}(m, p)\leq m^{m/p'}/m!$. Finally, we consider the case $L\in{\mathcal L}^{s}(^{m} L_{p}(\mu))$, $m'\leq p\leq m$. Using the estimate in \eqref{markov-ineq-1}, an induction on $m$ gives
\[
\|L\|\leq\frac{m^{m/m'}}{m!}\|\Lh\|=\frac{m^{(m-1)}}{m!}\|\Lh\|\,.
\]
Notice that for the induction argument in the case $m'\leq p\leq 2$ we need to consider $m'\leq p\leq (m-1)'$ and $(m-1)'\leq p\leq 2$, while in the case $2\leq p\leq m$ we need to consider $2\leq p\leq m-1$ and $m-1\leq p\leq m$, $m\geq 3$. We have proved the following result.

\begin{proposition}\label{polar-constants}
For the $m$th polarization constant $\mathbb{C}(m, p)$, $m\geq 2$, we have the estimates
\begin{equation}\label{polar-constants-1}
\mathbb{C}(m, p)\leq \begin{cases} \frac{m^{m/p}}{m!} &\text{$1\leq p\leq m'$ ,}\\
\frac{m^{m/m'}}{m!} &\text{$m'\leq p\leq m$ ,}\\
\frac{m^{m/p'}}{m!} &\text{$m\leq p\leq\infty$ .}
                                \end{cases}
\end{equation}
\end{proposition}

Using the polarization formula \eqref{polarformula} and inequality \eqref{WW-ineq} we can show that the estimates in \eqref{polar-constants-1} hold for \emph{complex} as well as for \emph{real} $L_{p}(\mu)$ spaces.

\begin{proposition}\label{gpolar-constants}
For the $m$th polarization constant $\mathbb{K}(m, p)$, $m\geq 2$, we have the estimates
\begin{equation}\label{gpolar-constants-1}
\mathbb{K}(m, p)\leq \begin{cases} \frac{m^{m/p}}{m!} &\text{$1\leq p\leq m'$ ,}\\
\frac{m^{m/m'}}{m!} &\text{$m'\leq p\leq m$ ,}\\
\frac{m^{m/p'}}{m!} &\text{$m\leq p\leq\infty$ .}
                                \end{cases}
\end{equation}
In the case $1\leq p\leq m'$ the estimate is best possible.
\end{proposition}

\begin{proof}
Let $x_{i}\in L_{p}(\mu)$, $1\leq i\leq m$, be unit vectors. From the polarization formula \eqref{polarformula} we have
\[
|L(x_1, \ldots, x_m)|\leq\frac{\|\Lh\|}{m!}\int_{0}^{1} \big\|\sum_{i=1}^{m} r_{i}(t)x_{i}\big\|_{p}^{m}\, dt\,.
\]
Since $1\leq p\leq m'\Leftrightarrow m\leq p'\leq \infty$, using H\"older's inequality first and then
inequality \eqref{WW-ineq}, the previous inequality gives the first estimate in \eqref{gpolar-constants-1}
(we also refer to the proof of Theorem $2$ in \cite{Sarantopoulos1}). The proof of the third estimate in \eqref{gpolar-constants-1} is similar. In particular, for $p=m$ we have
\[
\mathbb{K}(m, m)\leq\mathbb{K}(m, m')=\frac{m^{m/m'}}{m!}=\frac{m^{m-1}}{m!}\,.
\]
Finally, consider the case $m'\leq p\leq m$. Since $\mathbb{K}(m, p)$, as a function of $p$, is decreasing on the interval $[1, 2]$ and increasing for $p\geq 2$ (see \cite{Chatz-Sa}), for every $p\in [m', m]$ we have $\mathbb{K}(m, p)\leq m^{m/m'}/m!$.

To see that the estimate $m^{m/p}/m!$ is best possible in the case $1\leq p\leq m'$, we consider $L\in\mathcal{L}^{s}({}^{m}\ell_{p}; \mathbb{K})$ defined in Example \ref{markov-eq}. We have
$\|L\|\geq L(e_{1}, \ldots, e_{m})=1/m!$ and $\|\Lh\|=1/m^{m/p}$. Since by \eqref{gpolar-constants-1} $\mathbb{K}(m, p)\leq m^{m/p}/m!$, we conclude that $\|L\|= L(e_{1}, \ldots, e_{m})=1/m!$ and
\[
\|L\|=\frac{m^{m/p}}{m!}\|\Lh\|\,.
\]
Thus, $\mathbb{K}(m, p)= m^{m/p}/m!$.
\end{proof}

{\bf ($i$) Special case $m=2$.}

From \eqref{gpolar-constants-1} and for $1\leq p\leq\infty$ we have the estimate
\[
\mathbb{K}(2, p)\leq 2^{|p-2|/p}\,.
\]
For $1\leq p\leq 2$ we have $\mathbb{K}(2, p)=2^{(2-p)/p}$. In fact, in the case $1\leq p\leq m'$ it follows from Proposition \ref{gpolar-constants} that $\mathbb{K}(m, p)=m^{m/p}/m!$, for every $m\geq 2$. Therefore,
for $m=2$ and for $1\leq p\leq 2$ we have $\mathbb{K}(2, p)=2^{2/p}/2!=2^{(2-p)/p}$. To prove equality, as in Example \ref{markov-eq} we consider the $2$-homogeneous polynomial $\Lh(x)=x_{1}x_{2}$ with
$L(x, y)=\frac{1}{2}(x_{1}y_{2}+x_{2}y_{1})$, $x=(x_i)$, $y=(y_i)$, the corresponding symmetric bilinear form on the real or complex $\ell_{p}$ space, $1\leq p\leq 2$. Since $\|L\|\geq |L(e_{1}, e_{2})=1/2$ and for $x=(2^{-1/p}, 2^{-1/p}, 0, \ldots)$ with $\|x\|_{p}=1$, $\|\Lh\|=|\Lh(x)|=2^{-2/p}$, we have $\|L\|\geq 2^{(2-p)/p}\|\Lh\|$. But $\mathbb{K}(2, p)\leq 2^{(2-p)/p}$ and so $\|L\|=2^{(2-p)/p}\|\Lh\|$.

For $2\leq p\leq\infty$ we have $\mathbb{R}(2, p)=2^{(p-2)/p}$ and in particular $\mathbb{R}(2, \infty)=2$. To see this consider the $2$-homogeneous polynomial $\Lh(x)=x_{1}^{2}-x_{2}^{2}$ with
$L(x, y)=x_{1}y_{1}-x_{2}y_{2}$, $x=(x_i)$, $y=(y_i)$, the corresponding symmetric bilinear form on the \emph{real} $\ell_{p}$ space. Obviously $\|\Lh\|=|\Lh(e_{1})|=1$. On the other hand, for
$x=(2^{-1/p}, 2^{-1/p}, 0, \ldots)$ and $y=(2^{-1/p}, -2^{-1/p}, 0, \ldots)$ we have $\|x\|_{p}=\|y\|_{p}=1$ and $L(x, y)=2^{1-2/p}$. Hence, $\|L\|=2^{(p-2)/p}\|\Lh\|$.

{\bf ($ii$) Case $p\geq m\geq 3$.}

In this case the constant $\frac{m^{m/p'}}{m!}$ in \eqref{gpolar-constants-1} can be improved. It has been shown in \cite{Chatz-Sa} that for $p\geq m\geq 3$,
\begin{equation}\label{polar-1}
\mathbb{K}(m, p)\leq\frac{(2m)^{m/2}}{m!}\left(\frac{\Gamma\left(\frac{1}{2}(p+1)\right)}{\sqrt{\pi}}\right)^{m/p}\,,
\end{equation}
where $\Gamma$ is the \emph{gamma function}. For example, in the special case $p=m=4$ inequality \eqref{polar-1} gives
\[
\mathbb{K}(4, 4)\leq\frac{8^{2}}{4!}\cdot\frac{\Gamma(5/2)}{\sqrt{\pi}}=2\,.
\]
Since $p'=4/3$ is the conjugate exponent of $p=4$, from \eqref{gpolar-constants-1} we have the estimate $\mathbb{K}(4, 4)\leq\frac{4^3}{4!}=\frac{8}{3}$ which is bigger than $2$.

For the \emph{complex} $\ell_{\infty}$ Harris \cite[($16$)]{Harris1}, see also \cite[Proposition 1.43]{Dineen}, has shown that
\[
\mathbb{C}(m, \infty)=\mathbb{C}(m, \ell_{\infty})\leq\frac{m^{m/2}(m+1)^{(m+1)/2}}{2^{m}m!}
\]
and this upper estimate is smaller than $m^{m}/m!$. Tonge has also proved the same result by using a method very similar to the method which was used to prove that the {\em complex} Grothendieck constant $G(2)$ is bounded above by $\frac{3}{4}\sqrt{3}$, see \cite{Tonge1}.

\begin{remark}
Harris \cite[Theorem $6$]{Harris1} showed that if $1\leq p\leq\infty$ and $m$ is a power of $2$, then
\begin{equation}\label{Harris-est1}
\mathbb{C}(m, p)\leq\left(m^{m}/m!\right)^{|p-2|/p}\,.
\end{equation}
He has also conjectured that \eqref{Harris-est1} holds for all positive integers $m$ and that the constant given is best possible. But, as we have stated in Proposition \ref{gpolar-constants}, in the case $1\leq p\leq m'$, $m\geq 3$, the best constant is $\mathbb{C}(m, p)=m^{m/p}/m!$ and this is strictly less than $\left(m^{m}/m!\right)^{(2-p)/p}$. Observe that for $m=2$
\[
\mathbb{C}(2, p)=\frac{2^{2/p}}{2!}=2^{(2-p)/p}\,,\quad1\leq p\leq 2\,,
\]
and this is the constant given in \eqref{Harris-est1}.\\
On the other hand, for $m'\leq p\leq 2$, where $m=2^{n}$, $n\geq 2$, the constant given in \eqref{Harris-est1} has been improved in \cite[Theorem $3'$]{Sarantopoulos1}. But, in case $p$ is close to $2$, and for $m$ a power of $2$, Harris' bound is better than that of Proposition \ref{polar-constants}.
\end{remark}

%%%%%%%%%%%%%%%%%%%%%%%%%%%%%%%%%%%%%%%%%%%%%%%%%%%%%%%%%%%%%%%%%
%%%%%                                                       %%%%%
%%%%%                     SECTION 4                        %%%%%
%%%%%                                                       %%%%%
%%%%%%%%%%%%%%%%%%%%%%%%%%%%%%%%%%%%%%%%%%%%%%%%%%%%%%%%%%%%%%%%%

\section{\bf Bernstein's inequality for homogeneous polynomials on Hilbert spaces}

A famous result, investigated by Banach \cite{Banach} and many other authors, for example \cite{BS,CS,Harris1,Hormander,Kellogg,PST}, asserts that if $H$ is a Hilbert space, then $\mathbb{K}(n, H)=1$. In other words, $\|L\|=\|\Lh\|$ for every $L\in\mathcal{L}^{s}(^{n} H)$. Recall that $L$ is a continuous symmetric $n$-linear form on a Hilbert space $H$ and $\widehat{L}$ is the associated continuous $n$-homogeneous polynomial. Since the Fr\'echet derivative of $\Lh$ at $x\in H$ is given by $D\Lh(x)(y)=n L(x^{n-1} y)$, $y\in H$, where $L(x^{n-1} y):= L(\underbrace{x, \ldots , x}_{n- 1}, y)$, to prove $\|L\|=\|\Lh\|$ by an inductive argument, it suffices to show that $|L(x^{n- 1} y)|\leq\|\widehat{L}\|$ for any unit vectors $x$ and $y$ in $H$. In other words, $\|L\|=\|\Lh\|$ for any $\Lh\in\mathcal{P}\left({}^{n} H\right)$ if and only if
\begin{equation}\label{Bernstein-1}
\|D\Lh\|\leq n\|\Lh\|\,,\qquad \forall\, \Lh\in\mathcal{P}\left({}^{n} H\right)\,.
\end{equation}
Banach proved this result for continuous symmetric $n$-linear forms and continuous $n$-homogeneous polynomials on {\em finite dimensional real} Hilbert spaces. The proof works equally well for real and complex Hilbert spaces, and the condition of finite dimensionality is only needed to ensure that the $n$-linear form attains its norm.  The result that $\|L\|=\|\Lh\|$ is true for all Hilbert spaces, and, as pointed out by Banach, can be obtained through a simple limit argument based on the finite dimensional case.

Clearly, if $\Lh$ attains its norm at $x_{0}\in B_{H}$, the closed unit ball of the Hilbert space $H$, then $L$ also attains its norm at $\left(x_{0}, \ldots, x_{0}\right)\in B_{H}^{n}$. When $H$ is finite dimensional, $L$ will always attain its norm, since the closed unit ball of $H$ is compact. However, when $H$ is infinite dimensional, $L$ need not attain its norm: if $H=\ell_{2}$, the space of square summable sequences, and $L(x, y)=\sum_{n=1}^{\infty} \frac{n}{n+1}x_{n}y_{n}$, it is easy to see that $\|L\|=1$, but that $|L(x, y)|< 1$ for all unit vectors $x=(x_{n})$ and $y=(y_{n})$ in $H$.

It is true, but not obvious, that if $L$ attains its norm at $\left(x_{1}, \ldots, x_{n}\right)\in B_{H}^{n}$, then $\widehat{L}$ also attains its norm at some $x_{0}\in B_{H}$. When $L$ does attain its norm, an explicit construction has been given in \cite[section $2$]{PST} to provide a unit vector $x_{0}$ with $\|\Lh\|=|\Lh(x_{0})|$.

\begin{theorem}\cite[Theorem 2.1]{PST}\label{Banach-2}
If $L$ is a norm attaining continuous symmetric $n$-linear form, $n\geq 2$, on a Hilbert space, then the associated continuous symmetric $n$-homogeneous polynomial $\Lh$ also attains its norm. Moreover, $\|L\|=\|\Lh\|$.
\end{theorem}

For \textit{real} Hilbert spaces it is an interesting fact, see \cite[Theorem $4$]{Harris1}, that the Bernstein-type inequality \eqref{Bernstein-1} is equivalent to Szeg\"o's inequality for \textit{real} trigonometric polynomials (see \cite{CS}). That is, if $T(t)=\sum_{k=-n}^{n}c_{k}e^{ikt}$, $c_{-k}=\overline{c}_{k}$, is a real trigonometric polynomial of degree $n$ which satisfies $|T(t)|\leq 1$ for all real $t$, then
\begin{equation}\label{Szego-ineq}
n^{2}T(t)^{2}+ T^{\prime}(t)^{2}\leq n^{2}\,,\quad \forall\, t\in\mathbb{R}\,.
\end{equation}
But Szeg\"o's inequality \eqref{Szego-ineq} is a special case of a more general inequality for entire functions of exponential type. Recall that an entire function $f: \mathbb{C}\rightarrow \mathbb{C}$ is of \textit{exponential type} (EFET) if for some $A>0$ the inequality
\[
M_{f}(r):= \max_{|z|=r} |f(z)|< e^{Ar}
\]
holds for sufficiently large values of $r$. The greatest lower bound for those values of $A$ for which the latter asymptotic inequality is fulfilled is called the \textit{type $\sigma=\sigma_{f}$} of the function $f$. It follows from the definition of the type that
\[
\sigma_{f}=\limsup_{r\rightarrow \infty} \frac{\log{M_{f}(r)}}{r}\,.
\]
For example, if $ T(t)=\sum_{k=-n}^{n} c_{k} e^{ikt}$ is a trigonometric polynomial of degree $\leq n$, then $T(z)=\sum_{k=-n}^{n} c_{k}e^{ikz}$ is an EFET of type $\leq n$. A classical theorem due to Bernstein \cite{Bernstein2} states that if $f$ is an EFET of type $\leq \sigma$, then $f$ satisfies the inequality
\[
\sup_{t\in\mathbb{R}} |f^{\prime}(t)|\leq \sigma\sup_{t\in\mathbb{R}}|f(t)|\,.
\]
The following theorem, see \cite{Ac} or inequality $(11.4.5)$ in \cite{Boas}, contains Bernstein's inequality as a special case.

\begin{theorem}
Let $ f:\mathbb{C}\rightarrow \mathbb{C}$ be an entire function of exponential type $\leq\sigma$ and let $\sup_{t\in \mathbb{R}} |f(t)|<\infty$. Then for all $\omega\in\mathbb{R}$
\begin{equation}\label{Bernstein-2}
\sup_{t\in\mathbb{R}} |f^{\prime}(t)\cos\omega+ \sigma f(t)\sin\omega|\leq \sigma\sup_{t\in\mathbb{R}} |f(t)|\,.
\end{equation}
Equality holds in \eqref{Bernstein-2} if and only if $f(z)=ae^{i\sigma z}+ be^{-i\sigma z}$, where $a, b\in\mathbb{C}$.
\end{theorem}

In particular, if $T$ is a \textit{real} trigonometric polynomial of degree $n$ with $|T(t)|\leq 1$ for all real $t$, inequality \eqref{Bernstein-2} implies Szeg\"o's inequality \eqref{Szego-ineq}.

We prove now that the Bernstein-type inequality \eqref{Bernstein-1} on real or complex Hilbert spaces can be easily derived from inequality \eqref{Bernstein-2}(cf. \cite[Theorem 2.2]{AST}).

\begin{theorem}\label{Banach-2}
Let $\left(H, \langle\cdot, \cdot\rangle\right)$ be a real or complex Hilbert space. If $P: H\rightarrow \mathbb{K}$ is a continuous polynomial of degree $n$ and $x$ is a unit vector in $H$, then
\begin{equation}\label{ast-ineq1}
\{n^{2}|P(x)|^{2}-|DP(x)x|^{2}+\|DP(x)\|^{2}\}^{1/2}\leq n\|P\|\,.
\end{equation}
In particular, if $P=\Lh$ is a continuous $n$-homogeneous polynomial, then
\[
\|D\Lh\|\leq n\|\Lh\|\,.
\]
In other words, $\|L\|=\|\Lh\|$ for any $L\in\mathcal{L}^{s}\left(^{n} H\right)$.
\end{theorem}

\begin{proof}
Let $x, y$ be orthogonal unit vectors in $H$ and let $c\in\mathbb{K}$ satisfy $|c|=1$. Then $T(t):= P\left(x\cos{t} + cy\sin{t}\right)$ is a trigonometric polynomial of degree $\leq n$. But $\|x\cos{t}+ cy\sin{t}\|=1$ and therefore $|T(t)|\leq \|P\|$, for any $t\in\mathbb{R}$. Since  $T^{\prime}(t)= DP(x\cos{t}+ cy\sin{t})(-x\sin{t}+ cy\cos{t})$, Bernstein's inequality \eqref{Bernstein-2}, for $t=0$, implies
\[
|c DP(x)y\cos\omega+ nP(x)\sin\omega|\leq n\|P\|\,,\quad\forall\, \omega\in\mathbb{R}\,.
\]
By appropriate choice of $c$, $|c|=1$ and $\omega\in\mathbb{R}$ we get
\begin{equation}\label{ast-ineq2}
\{|DP(x)y|^{2}+ n^{2}|P(x)|^{2}\}^{1/2}\leq n\|P\|\,.
\end{equation}
Now, let $x$ be a fixed unit vector in $H$. Then, given a unit vector $u$ in $H$ it is possible to find a unit vector $y\in H$ orthogonal to $x$ so that $u=\alpha x+ \beta y$, where $|\alpha|^{2}+ |\beta|^{2}=1$. Since
\[
|DP(x)u)|^{2}=|\alpha DP(x)x+ \beta DP(x)y|^{2}\leq |DP(x)x|^{2}+ |DP(x)y|^{2}\,,
\]
using \eqref{ast-ineq2} we have
\[
\{|DP(x)u|^{2}- |DP(x)x|^{2}+ n^{2}|P(x)|^{2}\}^{1/2}\leq n\|P\|\,,\quad \forall u\in S_{H}\,.
\]
But $\|DP(x)\|=\sup_{\|u\|=1} |DP(x)u|$ and the proof of \eqref{ast-ineq1} follows.

If $P=\Lh$ is a continuous $n$-homogeneous polynomial, then as a particular case of \eqref{1-differential}  $D\Lh(x)x=n\Lh(x)$ and \eqref{ast-ineq1} is equivalent to $\|D\Lh(x)\|\leq n\|\Lh\|$, for every $x\in S_{H}$.
\end{proof}

In $1990$ Lomonosov \cite{Lomonosov} conjectured that Bernstein's inequality \eqref{Bernstein-1} for continuous $2$-homogeneous polynomials characterizes \textit{real} Hilbert spaces. Ben\'itez and Sarantopoulos \cite{Bensar} proved this conjecture in $1993$. In other words, it was shown that if $X$ is a \textit{real} Banach space, then $\|D\Lh\|\leq 2\|\Lh\|$ (or $ \|L\|=\|\Lh\|$) for any $\Lh\in\mathcal{P}\left({}^{2} X\right)$ if and only if $X$ is a \textit{real} Hilbert space.

However, Bernstein's inequality \eqref{Bernstein-1} for continuous homogeneous polynomials doesn't characterize \textit{complex} Hilbert spaces. As it has been proved in \cite{{Harris1}}, Bernstein's inequality for continuous homogeneous polynomials holds on the \textit{complex} $\ell_{\infty}^{2}$, the $2$-dimensional \textit{complex} $C(K)$ space. This result cannot be extended to all $C(K)$ spaces. For instance, in \cite{Tonge2} an example of a $2$-homogeneous polynomial was given on the \textit{complex} $\ell_{\infty}^{3}$ for which Bernstein's inequality fails. Recall that a $C(K)$ space is the Banach space of continuous functions on the compact Hausdorff space $K$, under the usual uniform norm. It is known that for any $\sigma$-finite measure $\mu$ the space $L_{\infty}(\mu)$ is isometric to a $C(K)$ space, see \cite[Proposition 4.3.8($ii$) and Theorem 4.3.7(Kelley \cite{Kelley})]{AK}. The simplest examples of $C(K)$ spaces are $\ell_{\infty}$ and
$L_{\infty}[0, 1]$.

Now we give another example of a \textit{complex} Banach space for which Bernstein's inequality for continuous homogeneous polynomials does hold. For this we need the following result of Harris.
\begin{proposition}\cite[Corollary $3$]{Harris1}
Let $\left(H, \langle \cdot, \cdot \rangle\right)$ be \emph{complex} Hilbert space and let $P: H\rightarrow \mathbb{C}$ be a continuous polynomial of degree
$n$. Then,
\begin{equation}\label{hcor3}
|nP(x)- DP(x)x|+ \|DP(x)\|\leq n\|P\|\,,\quad \forall x\in B_{H}\,.
\end{equation}
\end{proposition}

Observe that $S(x):= nP(x)- DP(x)x$ is the sum of the first $n-1$ partial sums of the polynomial $P$.

\begin{proposition}
If $H$ is a complex Hilbert space, consider the complex Banach space $H\times\mathbb{C}$, with the supremum norm,  which is a non-Hilbert space. Then,
\[
\|D\Lh\|\leq n\|\Lh\|\,,\qquad \forall\, \Lh\in\mathcal{P}\left({}^{n} H\times\mathbb{C}\right)\,.
\]
In other words, $\|L\|=\|\Lh\|$ for any $L\in\mathcal{L}^{s}\left(^{n} H\times\mathbb{C}\right)$.
\end{proposition}

\begin{proof}
Suppose $\dim(H)<\infty$. Any continuous $n$-homogeneous polynomial $\Lh$ on $H\times\mathbb{C}$ can be written in the form
\[
\Lh(<x, z>)=z^{n}P\left(\frac{x}{z}\right)\,,\quad \forall x\in H, z\in\mathbb{C}\,,
\]
where $P$ is a polynomial of degree $n$ on $H$. By the maximum modulus principle
\[
\|\Lh\|=\sup_{\|<x, z>\|=1} \left|z^{n}P\left(\frac{x}{z}\right)\right|=\sup_{\|x\|\leq 1} |P(x)|=\|P\|\,.
\]
To prove
\[
|D\Lh(<x, z>)<y, w>|\leq n\|\Lh\|\,,\quad \forall<x, z>,\, <y, w>\in B_{H\times\mathbb{C}}\,,
\]
by the maximum modulus principle is enough to show that
\[
|D\Lh(<x, 1>)<y, 1>|\leq n\|\Lh\|\,,\quad \forall x,\, y\in B_{H}\,.
\]
For this we need the following identity, which can be easily checked
\[
D\Lh(<x, 1>)<y, 1>=DP(x)y +nP(x)- DP(x)x\,.
\]
Then, from inequality \eqref{hcor3} it follows that
\[
|D\Lh(<x, 1>)<y, 1>|\leq n\|P\|=n\|\Lh\|\,,\quad \forall x,\, y\in B_{H}\,.
\]
Based on the finite dimensional case, a simple argument gives the proof in the case $H$ is an arbitrary complex Hilbert space.
\end{proof}

Observe that in the special case $H=\mathbb{C}$, the space $H\times\mathbb{C}$ with the supremum norm is just the complex space $\ell_{\infty}^{2}$.

\begin{problem}
Characterize the \textit{complex} Banach spaces $X$ for which Bernstein's inequality holds for any continuous homogeneous polynomial on $X$. That is, the \textit{complex} Banach spaces $X$ which share the property
\[
\|D\Lh\|\leq m\|\widehat{L}\|\Leftrightarrow\|L\|=\|\Lh\|\,,\quad \forall\,\Lh\in\mathcal{P}\left({}^{m} X\right)\,.
\]
\end{problem}


\begin{thebibliography}{99}
\bibitem{AAP}\label{AAP}
{\sc M. D. Acosta, F. Aguirre \and R. Pay\'{a}}, \emph{There is no bilinear Bishop-Phelps theorem}, Israel J. Math. {\bf 93} (1996), 221--227.
%
\bibitem{Ac}\label{Ac}
{\sc N.~I. Achieser}, \emph{Theory of approximation}, Dover Publ., 1992.
%
\bibitem{AK}\label{AK}
{\sc F. Albiac and N. J. Kalton}, \emph{Topics in Banach Space Theory}, Springer, 2006.
%
\bibitem{AST}\label{AST}
{\sc V. A. Anagnostopoulos, Y. Sarantopoulos \and A. M.  Tonge}, \emph{Homogeneous polynomials and extensions of Hardy-Hilbert's inequality}, Math. Nachr. {\bf 97}, No. 1, (2012), 47--55.
%
\bibitem{Banach}\label{Banach}
{\sc M. Banach}, \emph{\"{U}ber homogene Polynome in $(L^2)$}, Studia Math. {\bf 7} (1938), 36--44.
%
\bibitem{BARAN1}\label{BARAN1}
{\sc M. Baran}, \emph{Bernstein type theorems for compact sets in $\mathbb{R}^n$ revisited}, J. Approx. Th. {\bf 79} (1994), 190--198.
%
\bibitem{BARAN2}\label{BARAN2}
{\sc M. Baran and W. Ple\'sniak}, \emph{Bernstein and van der Corput-Schaake type inequalities on semialgebraic curves}, Studia Math. {\bf 125}(1) (1997), 83--96.
%
\bibitem{BARAN3}\label{BARAN3}
{\sc M. Baran}, \emph{Polynomial inequalities in Banach spaces}, Banach Center Publ., {\bf 107}, Polish Acad. Sci. Inst. Math., Warsaw, 2015.
%
\bibitem{Bensar}\label{Bensar}
{\sc C. Ben\'itez and Y. Sarantopoulos}, \emph{Characterization of real inner product spaces by means of symmetric bilinear forms}, J. Math. Anal. Appl. {\bf 180} (1993), 207--220.
%
\bibitem{Bernstein1}\label{Bernstein1}
{\sc S.~N. Bernstein}, \emph{Sur l'ordre de la meilleure approximation des fonctions continues par des polynomes de degr\'e donn\'e}, Memoires de l'Acad\'emie Royale de Belgique, {\bf 4} (1912), 1--103.
%
\bibitem{Bernstein2}\label{Bernstein2}
---------, \emph{Collected works: Vol. I. Constructive theory of functions (1905-1939)}, English translation, {\em Atomic Energy Commission, Springfield, VA,} 1958.
%
\bibitem{Bialas}\label{Bialas}
{\sc L. Bialas-Cie\^z and P. Goetgheluck}, \emph{Constants in Markov's inequality on convex sets}, East J. Approx. {\bf1} (1995), 379--389.
%
\bibitem{BP}\label{BP}
{\sc E. Bishop  \and R. R. Phelps}, \emph{A proof that every Banach space is subreflexive}, Bull. of Amer. Math. Soc. {\bf 67} (1961), 97--98.
%
\bibitem{Boas}\label{Boas}
{\sc R.~P. Boas}, \emph{Entire functions}, Academic Press, 1954.
%
\bibitem{BS}\label{BS}
{\sc J. Bochnak and J. Siciak}, \emph{Polynomials and multilinear mappings in topological vector spaces}, Studia Math. {\bf 39} (1971), 59--76.
%
\bibitem{BORWEIN}\label{BORWEIN}
{\sc P. B. Borwein and T. Erd\'elyi}, \emph{Polynomials and polynomial inequalities}, Springer Verlag, New York, 1995.
%
\bibitem{BOS}\label{BOS}
{\sc L. Bos, N. Levenberg and B. A. Taylor}, \emph{Characterization of smooth, compact algebraic curves in $\mathbb{R}^2$, in: Topics in Complex Analysis, P. Jak\'obczak and W. Ple\'sniak(eds.), Banach Center Publ. 31, Inst. Math., Polish Acad. Sci., Warszawa}, 1995, 125--134.
%
\bibitem{BLMR}\label{BLMR}
{\sc D. Burns, N. Levenberg, S. Ma'u and S. R\'ev\'esz}, \emph{Monge-Amp\'ere measures for convex bodies and Bernstein-Markov type inequalities}, Trans. Amer. Math. Soc. {\bf 362}, no. 12, (2010), 6325--6340.
%
\bibitem{Chae}\label{Chae}
{\sc S. B. Chae}, \emph{Holomorphy and calculus in normed spaces}, Dekker, New York, 1985.
%
\bibitem{Chatz-Sa}\label{Chatz-Sa}
{\sc M. Chatzakou and Y. Sarantopoulos}, \emph{Polarization constants for polynomials on $L_{p}(\mu)$ spaces}, preprint, 2019.
%
\bibitem{CS}\label{CS}
{\sc J. G. van der Corput and G. Schaake}, \emph{Ungleichungen f\"ur polynome und trigonometrische polynome}, Compositio Math. {\bf 2} (1935) 321--361; Berichtigung, {\em Compositio Math. } {\bf 3} (1936), 128.
%
\bibitem{Davie-1}\label{Davie-1}
{\sc A. M. Davie}, \emph{Quotient algebras of uniform algebras}, J. London Math. Soc. (2) {\bf 7} (1973), 31--40.
%
\bibitem{Davie-2}\label{Davie-2}
---------, \emph{Power bounded elements in a $Q$-algebra}, Bull. London Math. Soc. {\bf 6} (1974), 61--65.
%
\bibitem{Dineen}\label{Dineen}
{\sc S. Dineen}, \emph{Complex analysis on infinite dimensional spaces}, Springer Monographs in Mathematics, Springer-Verlag, Berlin, 1999.
%
\bibitem{DUFFIN1}\label{DUFFIN1}
{\sc A. C. Schaeffer and R. J. Duffin}, \emph{On some inequalities of S. Bernstein and W. Markoff}, Bull. Amer. Math. Soc. {\bf 44} (1938), 289--297.
%
\bibitem{EGGLESTON}\label{EGGLESTON}
{\sc H. G. Eggleston}, \emph{Convexity}, Cambridge Tracts in Math. 47, Cambridge Univ. Press, 1977.
%
\bibitem{EI-1}\label{EI-1}
{\sc A. Eskenazis and P. Ivanisvili}, \emph{Dimension independent Bernstein-Markov inequalities in Gauss space}, arXiv: 1808.01273 (2018).
%
\bibitem{EI-2}\label{EI-2}
---------, \emph{Polynomial inequalities on the Hamming cube}, arXiv: 1902.02406 (2019).
%
\bibitem{Fefferman}\label{Fefferman}
{\sc Ch. Fefferman and R. Narasimhan}, \emph{Bernstein's inequality and the resolution of spaces of analytic functions}, Duke Math. J. {\bf 81} (1995), 77--98.
%
\bibitem{Harris1}\label{Harris1}
{\sc L. A. Harris}, \emph{Bounds on the derivatives of holomorphic functions of vectors} {\em in: Colloque d'Analyse (Rio de Janeiro, 1972),  L. Nachbin (ed), Actualit\'{e}s Sci. Indust.} 1367, Hermann, Paris (1975) 145--163.
%
\bibitem{Harris2}\label{Harris2}
---------, \emph{A Bernstein-Markov theorem for normed spaces}, J. Math. Anal. Appl. {\bf 208} (1997), 476--486.
%
\bibitem{Harris3}\label{Harris3}
---------, \emph{Multivariate Markov polynomial inequalities and Chebyshev nodes}, J. Math. Anal. Appl. {\bf 338} (2008), 350--357.
%
\bibitem{Harris4}\label{Harris4}
---------, \emph{A proof of Markov's theorem for polynomials on Banach spaces}, J. Math. Anal. Appl. {\bf 368} (2010), 374--381.
%
\bibitem{Hormander}\label{Hormander}
{\sc L. H\"{o}rmander}, \emph{On a theorem of Grace}, Math. Scand. {\bf 2} (1954) 55--64.
%
\bibitem{Kelley}\label{Kelley}
{\sc J. L. Kelley}, \emph{Banach spaces with the extension property}, Trans. Amer. Math. Soc. {\bf 72}, (1952), 323--326.
%
\bibitem{Kellogg}\label{Kellogg}
{\sc O. D. Kellogg}, \emph{On bounded polynomials in several variables}, Math. Zeit. {\bf 27} (1928), 55--64.
%
\bibitem{KR}\label{KR}
{\sc A. Kro\'o and S. R\'ev\'esz}, \emph{On Bernstein and Markov type inequalities for multivariate polynomials on convex bodies}, J. Approx. Theory, {\bf 99} (1999), 134--152.
%
\bibitem{Kroo1}\label{Kroo1}
{\sc A. Kro\'o}, \emph{On Markov inequality for multivariate polynomials}, in: Approximation, vol. XI, Gatlinburg, 2004, Nashboro Press, Brentwood, TN, 2005, 211--227.
%
\bibitem{Kroo2}\label{Kroo2}
---------, \emph{On Bernstein-Markov-type inequalities for multivariate polynomials in $L_q$-norm}, J. Approx. Theory {\bf 159} (2009), 85--96.
%
\bibitem{Lomonosov}\label{Lomonosov}
{\sc V. Lomonosov}, \emph{Personal communication}, 1990.
%
\bibitem{AMARKOV}\label{AMARKOV}
{\sc A. A. Markov}, \emph{On a problem of D. I. Mendeleev}, Zap. Im. Akad. Nauk. {\bf 62} (1889), 1--24.
%
\bibitem{VMARKOV}\label{VMARKOV}
{\sc V. A. Markov}, \emph{\"Uber Polynome, die in einen gegebenen Intervalle m\"oglichst wenig von Null abweichen}, Math. Ann. {\bf 77} (1916), 213--258.
%
\bibitem{Gustavo1}\label{Gustavo1}
{\sc G. A. Mu\~noz, Y. Sarantopoulos and A. Tonge}, \emph{Complexifications of real Banach spaces, polynomials and multilinear maps}, Studia Math. {\bf 134} (1999), 1--33.
%
\bibitem{Gustavo2}\label{Gustavo2}
{\sc G. A. Mu\~noz and Y. Sarantopoulos}, \emph{Bernstein and Markov-type inequalities for polynomials on real Banach spaces}, Math. Proc. Camb. Phil. Soc. {\bf 133} (2002), 515--530.
%
\bibitem{PST}\label{PST}
{\sc A. Pappas, Y. Sarantopoulos and A. Tonge}, \emph{Norm attaining polynomials}, Bull. London Math. Soc. \textbf{39} (2007), 255--264.
%
\bibitem{Peller}\label{Peller}
{\sc V. V. Peller}, \emph{Estimates of functions of power bounded operators on Hilbert spaces}, J. Operator Theory {\bf 7} (1982), 341--372.
%
\bibitem{Plesniak}\label{Plesniak}
{\sc W. Ple\'sniak}, \emph{Recent progress in multivariate Markov inequality}, in: Monogr. Textbooks Pure Applied Math., vol. 212, Dekker, New York, 1998, 449--464.
%
\bibitem{Rahman}\label{Rahman}
{\sc Q. I. Rahman and G. Schmeisser}, \emph{Analytic theory of polynomials}, London mathematical society monographs. New series, 26; Oxford science publications. Oxford: Clarendon Press, 2002.
%
\bibitem{RevSar}\label{RevSar}
{\sc Sz. R\'ev\'esz and Y. Sarantopoulos}, \emph{On Markov constants of homogeneous polynomials over real normed spaces}, East J. Approx. {\bf 9} (2003), no. 3, 277--304.
%
\bibitem{Revesz}\label{Revesz}
{\sc Sz. R\'ev\'esz}, \emph{Conjectures and results on the multivariate Bernstein inequality on convex bodies}, in: Constructive theory of functions, 318--353, Prof. M. Drinov Acad. Publ. House, Sofia, 2012.
%
\bibitem{Sarantopoulos1}\label{Sarantopoulos1}
{\sc Y. Sarantopoulos}, \emph{Estimates for polynomial norms on $L^{p}(\mu)$ spaces}, Math. Proc. Camb. Phil. Soc. {\bf 99} (1986), 263--271.
%
\bibitem{Sarantopoulos2}\label{Sarantopoulos2}
---------, \emph{Bounds on the derivatives of polynomials on Banach spaces}, Math. Proc. Camb. Phil. Soc. {\bf 110} (1991), 307--312.
%
\bibitem{Skalyga1}\label{Skalyga1}
{\sc V. I. Skalyga}, \emph{Bounds for the derivatives of polynomials on centrally symmetric convex bodies}, Izv. Math. {\bf 69} (3) (2005), 607--621.
%
\bibitem{Skalyga2}\label{Skalyga2}
---------, \emph{V. A. Markov's theorems in normed spaces}, Izv. Math. {\bf 72} (2008), 383--412.
%
\bibitem{Taylor}\label{Taylor}
{\sc A. E. Taylor}, \emph{Additions to the theory of polynomials in normed linear spaces}, T\^ohoku Math. J., {\bf 44} (1938), 302--318.
%
\bibitem{Tonge1}\label{Tonge1}
{\sc A. M. Tonge}, \emph{Polarization and the two-dimensional Grothendieck inequality}, Math. Proc. Camb. Phil. Soc. {\bf 95} (1984), 313--318.
%
\bibitem{Tonge2}\label{Tonge2}
---------, \emph{The failure of Bernstein's theorem for polynomials on $C(K)$ spaces}, J. Approx. Theory {\bf 51} (1987), 160--162.
%
\bibitem{Varopoulos}\label{Varopoulos}
{\sc N. TH. Varopoulos}, \emph{On a commuting family of contractions on a Hilbert space}, Rev. Roumaine Math. Pures Appl. {\bf 21} (1976), 1283--1285.
%
\bibitem{Wilhelmsen}\label{Wilhelmsen}
{\sc D. R. Wilhelmsen}, \emph{A Markov inequality in several dimensions}, J. Approx. Theory {\bf 11} (1974), 216--220.
%
\bibitem{WW}\label{WW}
{\sc L. R. Williams and J. H. Wells}, \emph{$L^p$ Inequalities}, J. Math. Anal. Appl. {\bf 64} (1978), 518--529.
%
\end{thebibliography}
\end{document}